\documentclass{amsart}

\begin{document}
\theoremstyle{plain}
\newtheorem{theorem}{Theorem}
\newtheorem{corollary}[theorem]{Corollary}
\newtheorem{lemma}[theorem]{Lemma}
\newtheorem{proposition}[theorem]{Proposition}
\newtheorem{identity}[theorem]{Identity}

\theoremstyle{definition}
\newtheorem{definition}[theorem]{Definition}
\newtheorem{example}[theorem]{Example}
\newtheorem{conjecture}[theorem]{Conjecture}

\theoremstyle{remark}
\newtheorem{remark}[theorem]{Remark}
\newtheorem{note}[theorem]{Note}
\title{Some Formulas for Numbers
   of  Restricted Words}
\author{
Milan Janji\'c}
\address{\textsuperscript{1}Department of Mathematics and Informatics. University of Banja Luka. Republic of Srpska, BA
.}
\email{agnus@blic.net}
\date{\today}
\begin{abstract}
We define a quantity $c_m(n,k)$ as a generalization of the notion of the composition of the positive integer $n$ into $k$ parts. We proceed to derive  some known  properties of this quantity. In particular, we relate two partial Bell polynomials, in which the sequence of the variables of one polynomial is the invert transform of the sequence of the variables of the other.

We connect the quantities $c_m(n,k)$ and $c_{m-1}(n,k)$ via Pascal matrices. We then  relate $c_m(n,k)$ with the numbers of some restricted words over a finite alphabet.
 We develop a method which transfers some properties of restricted words over an alphabet of $N$ letters to  the restricted words over an alphabet of  $N+1$ letters.
  Several examples illustrate our findings.

Note that all our results depend solely on the initial arithmetic function $f_0$.
\end{abstract}
\maketitle
\section{Introduction}
 In this paper, the author continues the investigation of the properties of restricted words, introduced in two previous papers.
 In  the author's paper~\cite{ja1}, for a given initial arithmetic function $f_0$, we investigated functions  $f_1,f_2,\ldots$ such that $f_m$ is the $m$th inverse transform of $f_0$.  This paper, and Birmajer et al.~\cite{bir1}, consider some cases in which $f_m$ counts the number of restricted words over a finite alphabet.
 In the present paper, we considers the function $c_m(n,k)$ which, in a natural way,  generalizes the notion of the composition of $n$ into $k$ parts. Similar functions have recently  been considered by several authors, e.g., Eger~\cite{eg}.

Let $f_0$ be an arithmetic function. For the positive integer $m$, we let $f_m$ denote the $m$th invert transform of $f_0$. For non-negative integers $n,k,(1\leq k\leq n)$, we
define $c_m(n,k)$ recursively in the following way:\[c_m(0,0)=1,\;c_m(n,0)=0,(n>1),\] and \begin{equation}\label{e2}c_m(n,k)=\sum_{i=1}^{n-k+1}f_{m-1}(i)c_m(n-i,k-1),(1\leq k\leq
n).\end{equation}
Note that throughout the  paper letters $m,n,k$ will have the meaning as in this definition of $c_m(n,k)$. The definition is a natural generalization of the recurrence for the number of  the standard compositions of $n$ into $k$ parts. Firstly, we derive some  basic properties of $c_m(n,k)$. As a consequence, we  get  a  relation between the partial Bell polynomial of variables $x_1,x_2,\ldots$ and the partial  Bell  polynomial of variables $y_1,y_2,\ldots$, when the second sequence of variables is the invert transform of the first. Using Birmajer et al.~\cite[Corollary 10]{bir1}, we  derive a formula connecting $c_m(n,k)$ with $c_{m-1}(n,k)$. The formula may simply  be written with the use of the lower triangular Pascal matrices.  We then  extend this result to obtain a relation between $c_m(n,k)$ and $c_1(n,k)$.

 For the particular case when $f_0(1)=1$, we develop a method which allows us to derive an interpretation of $c_m(n,k)$ in terms of restricted words, when we know the restricted words counted by $c_1(n,k)$.
 We finish the paper with a number of examples illustrating our results.
\section{Some basic properties of $c_m(n,k)$}

We start with a few simple facts.
For $n=k$, we have $c_m(n,n)=f_{m-1}(1)c_m(n-1,n-1)$. Continuing the same procedure, we obtain
\begin{equation}\label{eq2} c_m(n,n)=f_{m-1}(1)^n.\end{equation}
 For $k=1$, we have
$c_m(n,1)=\sum_{i=0}^nf_{m-1}(i)g(n-i,0)$, that is,
\begin{equation}\label{eq3}c_m(n,1)=f_{m-1}(n).\end{equation}

We next give a simple proof of the fact that $c_m(n,k)$ may be expressed in terms of the partial Bell  polynomials $B_{n,k}$.
\begin{proposition}\label{pr1} It is true that
\begin{equation}\label{bpp}c_m(n,k)=\frac{k!}{n!}B_{n,k}(1!f_{m-1}(1),2!f_{m-1}(2),\ldots).\end{equation}
\end{proposition}
\begin{proof}We write Equation (\ref{e2}) in the following way:
 \[\frac{n!}{k!}c_m(n,k)=\frac{1}{k}\sum_{i=1}^{n-k+1}f_{m-1}(i)\frac{n!}{(k-1)!}c_m(n-i,k-1).\]
 It follows that
\[\frac{n!}{k!}c_m(n,k)=\frac{1}{k}\sum_{i=1}^{n-k+1}i!f_{m-1}(i){n\choose
i}\frac{(n-i)!}{(k-1)!}c_m(n-i,k-1).\]
Hence, the  quantity $b(n,k)=\frac{n!}{k!}c_m(n,k)$
satisfies the following recurrence:
\[kb(n,k)=\sum_{i=1}^{n-k+1}i!f_{m-1}(i){n\choose i}b(n-i,k-1).\]
Denoting $i!f_{m-1}(i)=x_i,(i=1,2,\ldots)$, we obtain the well-known recurrence for the partial Bell polynomials:
\[kB_{n,k}(x_1,x_2,\ldots)=\sum_{i=1}^{n-k+1}{n\choose i}x_iB_{n-i,k-1}(x_1,x_2,\ldots).\]
\end{proof}

We next  prove that $c_m(n,k)$ is a convolution of the sequence
$f_{m-1}(1),f_{m-1}(2),\ldots$.
\begin{proposition} We have
\begin{equation}\label{eger}
c_m(n,k)=\sum_{i_1+i_2+\cdots+i_k=n}f_{m-1}(i_1)f_{m-1}(i_2)\cdots f_{m-1}(i_k),
\end{equation}
where the sum is taken over positive $i_t$ for $t=1,2,\ldots,k$.
\end{proposition}
\begin{proof} Since
\begin{gather*}\sum_{i_1+i_2+\cdots+i_k=n}f_{m-1}(i_1)f_{m-1}(i_2)\cdots f_{m-1}(i_k)\\=
\sum_{i_k=1}^{n-k+1}f_{m-1}(i_k)\sum_{i_1+i_2+\cdots+i_{k-1}=n-i_k}f_{m-1}(i_1)f_{m-1}(i_2)\cdots
f_{m-1}(i_{k-1}),\end{gather*}
the proof easily follows by induction on $k$.
\end{proof}
\begin{remark}
 Equation (\ref{eger}) links $c_m(n,k)$ and the weighted integer compositions defined  in Eger~\cite{eg}.
 Here, the extended binomial coefficients ${k\choose n}_f$ are defined to count the number of the so-called weighted compositions of $n$ into $k$ parts, where
$f$ is a weighted function. Extended binomial coefficients  and $c_m(n,k)$ are connected by
\[c_m(n,k)={k\choose n-k}_{f_{m-1}}.\]
  \end{remark}

We know that the number of all compositions of $n$ equals the sum of compositions of $n$ into $k$ parts. For the functions  $f_m(n)$ and $c_m(n,k)$, we prove the analogous result.
 \begin{proposition}\label{pr3} The following equation is true:
\[f_m(n)=\sum_{k=1}^nc_m(n,k).\]
\end{proposition}
\begin{proof}
We use induction on $n$. For $n=1$, we have $f_m(1)=c_m(1,1)$.
From Equation (\ref{eq2}), it follows that $c_m(1,1)=f_{m-1}(1)$. Equation $f_m(1)=f_{m-1}(1)$ follows from the author~\cite[Corollary 2]{ja1}.

Assume that the claim is true  for $n-1$. It follows that
\[\sum_{k=1}^nc_m(n,k)=\sum_{k=1}^{n}\sum_{i=1}^{n-k+1}f_{m-1}(i)c_m(n-i,k-1).\]
Changing the order of summation on the right-hand side  implies
\[\sum_{k=1}^nc_m(n,k)=\sum_{i=1}^{n}f_{m-1}(i)\sum_{k=1}^{n-i+1}c_m(n-i,k-1).\]
We thus obtain
\[\sum_{k=1}^nc_m(n,k)=f_{m-1}(n)+\sum_{i=1}^{n-1}f_{m-1}(i)\sum_{k=1}^{n-i+1}c_m(n-i,k-1).\]
In the second sum on the right-hand side,  the term obtained for $k=1$  equals
zero, which implies that
\[\sum_{k=1}^nc_m(n,k)=f_{m-1}(n)+\sum_{i=1}^{n-1}f_{m-1}(i)\sum_{k-1=1}^{n-i}c_m(n-i,k-1).\]
Using the induction hypothesis, we obtain
\[\sum_{k=1}^nc_m(n,k)=f_{m-1}(n)+\sum_{i=1}^{n-1}f_{m-1}(i)f_m(n-i),\]
which proves the assertion.
\end{proof}
\begin{remark} The statement of this proposition is obvious for the standard compositions. In our case, it depends on arbitrary initial arithmetic function $f_0$. So, we needed a formal proof.
\end{remark}
\section{A connection of  $c_m(n,k)$ and $c_{m-1}(n,k)$}
We may view the array $c_m(n,k),(1\leq k\leq n)$ as a lower triangular matrix $C_m(n)$ of order $n$, whose $(n,k)$ entry equals $c_m(n,k)$. We let $L_n$ denote the lower triangular Pascal matrix. Hence, the  $(n,k)$ entry of $L_n$ is ${n-1\choose k-1},(1\leq k\leq n)$.

First, we prove the following:
\begin{proposition} For each $m>1$, we have
\[C_m(n)=C_{m-1}(n)\cdot L_n.\]
\end{proposition}
\begin{proof}
 It is easy to see that the statement  is equivalent to the following equation:
\begin{equation}\label{cmk}c_m(n,k)=\sum_{i=k}^n{i-1\choose k-1}c_{m-1}(n,i).\end{equation}

In our terminology, Birmajer et al.~\cite[Corollary 10]{bir2} may be written in the form
\[\sum_{j_1+j_2+\cdots+j_k=n}f_m(j_1)\cdots f_m(j_k)=\sum_{i=1}^n{i+k-1\choose
i}c_{m-1}(n,i),\]
where the sum is taken over nonnegative $j_1,\ldots,j_k$. Since at most $k-1$ of $j_t$ may equal $0$, Equation (\ref{eger}) yields
\[\sum_{j=0}^{k-1}{k\choose j}c_m(n,k-j)=\sum_{i=1}^n{i+k-1\choose
i}c_{m-1}(n,i).\]
Replacing $k-j$ by $t$ and denoting $\sum_{i=1}^n{i+k-1\choose
i}c_{m-1}(n,i)=a_k$ implies
\[\sum_{t=1}^{k}{k\choose t}c_m(n,t)=a_k,(k=1,2,\ldots,n).\]

Denoting $X=(c_m(n,1),c_m(n,2),\ldots,c_m(n,n))^T$, and
$A=(a_1,a_2,\ldots,a_n)^T$, this system may be written in the matrix form
\[Q\cdot X=A,\] where $Q$ is obtained from the Pascal matrix $L_{n+1}$ by omitting the first row and the first column.
It follows that $X=Q^{-1}\cdot A$, where
$Q^{-1}=\left((-1)^{i+j}{i\choose j}\right)_{n\times n}.$
For $k=1,2,\ldots,n$, we obtain
\begin{equation}\label{b2} c_m(n,k)=\sum_{i=1}^n\left[\sum_{j=1}^n(-1)^{j+k}{k\choose
j}{i+j-1\choose i}\right]c_{m-1}(n,i). \end{equation}

Formula  (\ref{b2}) holds for each $m> 1$ and for an arbitrary arithmetic function $f_0$. In particular,taking  $f_0(1)=1,f_0(i)=0,(i>1)$, we obviously have  $c_1(n,n)=1$, and $c_1(n,k)=0$ for $k<n$. Also
$f_1(n)=1$ for all $n$.
In this case, $c_2(n,k)$ equals the number of compositions of $n$ into $k$ parts, that is,
$c_2(n,k)={n-1\choose k-1}$.
Therefore, Equation (\ref{b2}) becomes
\begin{equation}\label{id}{n-1\choose k-1}=\sum_{j=1}^k(-1)^{j+k}{k\choose j}{n+j-1\choose n}.\end{equation}
 The expression in the square brackets in  Equation (\ref{b2}) equals ${i-1\choose k-1}$, which proves that Equation (\ref{e3}) is true.
\end{proof}
\begin{remark} Note that, as a byproduct, we proved the binomial identity (\ref{id}).
\end{remark}
\begin{remark}
Replacing $i-k$ by $t$ in Equation (\ref{cmk}), we obtain
\begin{equation}\label{ekf}c_m(n,k)=\sum_{t=0}^{n-k}{k+t-1\choose
t}c_{m-1}(n,k+t).\end{equation}
\end{remark}

From the equation $C_m(n,k)=C_{m-1}(n,k)\cdot L_n$ follows
\[C_m(n,k)=C_{m-1}(n,k)\cdot L_n=C_{m-2}(n,k)\cdot L_n^2=\cdots=C_1(n)\cdot L_n^{m-1}.\]
We thus obtain
\begin{proposition}\label{vv} The following matrix equation holds \[C_m(n)=C_1(n)L_n^{m-1},\]
or, explicitly,
\begin{equation}\label{e3}c_m(n,k)=\sum_{i=k}^n(m-1)^{i-k}{i-1\choose k-1}c_1(n,i),\;(1\leq k\leq n).\end{equation}
\end{proposition}
\begin{proof} Since entries of $L_n^{m-1}$ are well known, we easily obtain Equation (\ref{e3}).
\end{proof}
Now, we  derive a formula in which $f_m(n)$ is expressed in terms of $c_1(n,k)$.
\begin{proposition}
The following formula holds
\begin{equation}\label{fmn1}f_m(n)=\sum_{i=1}^n m^{i-1}c_1(n,i).\end{equation}
\end{proposition}
\begin{proof}
Proposition \ref{pr3} yields
\[f_m(n)=\sum_{k=1}^n\sum_{i=k}^n(m-1)^{i-k}{i-1\choose k-1}c_1(n,i).\]
Changing the order of summation gives
\[f_m(n)=\sum_{i=1}^n\left[\sum_{k=1}^i(m-1)^{i-k}{i-1\choose k-1}\right]c_1(n,i).\]
Using the binomial theorem, we obtain Equation (\ref{fmn1}).
\end{proof}
\begin{note}
Equation (\ref{fmn1}) appears in~Birmajer et al.~\cite{bir2}, where it is obtained using the properties of the partial Bell polynomials.
\end{note}
As an immediate consequence of Proposition \ref{pr1} and Equation (\ref{e3}), we obtain the following identity for the partial Bell polynomials.
\begin{identity}
If the sequence $y_1,y_2,\ldots$ is the invert transform of the sequence $x_1,x_2,\ldots$,
then
\[k!B_{n,k}(y_1,2!\cdot y_2,3!\cdot y_3,\ldots)=\sum_{i=k}^n{i-1\choose
k-1}i!B_{n,i}( x_1,2!\cdot x_2,3!\cdot x_3, \ldots).\]
\end{identity}

The following simple result connects the number of some compositions of $n$ into $k$ parts
with the number of restricted binary words of length $n-1$ with $k-1$ ones. The restriction, which we denote by $\mathcal R$,  is given by the kind of compositions.
\begin{corollary}There is a bijection between compositions of $n$ into $k$ parts and $\mathcal R$-restricted binary words of length $n-1$ with $k-1$ ones.
\end{corollary}
\begin{proof}
The bijection is given by the following correspondence:
\begin{equation}\label{kor}1\to 1,2\to 10,3\to 100,\ldots.\end{equation}
In this way, the compositions of $n$ into $k$ parts designate  $\mathcal R$-restricted binary word of length $n$ and  having $k$ ones, all of which begin with $1$. The converse  is also true. Omitting the leading $1$, we obtain the desired correspondence.
\end{proof}

We next extend the above result.
\begin{proposition}\label{hr}
Take  $m>1$. Assume that $f_{0}(1)=1$ and that $f_{m-1}(n)$ counts the $\mathcal R$-restricted  words of length $n-1$ over a finite alphabet $\alpha$.
Let $x$ be a letter which is not in $\alpha$.
Then, $c_m(n,k)$ equals the number of $\mathcal R$-restricted words of length $n-1$ over the alphabet $\alpha\cup\{x\}$ in which $x$ appears $k-1$ times.
\end{proposition}
\begin{proof}
Since $f_0(1)=1$, it follows from the author~\cite[Corollary 2]{ja1}  that $f_{m-1}(1)=1$. From this, one easily obtains that $f_{m-1}(n)>0$ for all $n$.

We use induction on $k$. For $k=1$, Equation (\ref{eq3}) yields $c_m(n,1)=f_{m-1}(n).$
Since $f_{m-1}(n)$ equals the number of $\mathcal R$-restricted words of length $n-1$ over $\alpha$, not containing $x$, we conclude that the statement is true for $k=1$.
 Assume that it is true for $k-1$. Consider the first term $f_{m-1}(1)c_m(n-1,k-1)=c_m(n-1,k-1)$ in Equation (\ref{e2}).  By the induction hypothesis, $c_{m}(n-1,k-1)$ equals the number of $\mathcal R$-restricted words of length $n-2$ having $k-2$ letters equal to $x$. Adding $x$ at the beginning of each such word, we obtain all $\mathcal R$-restricted words of length $n-1$ over $\alpha\cup\{x\}$,  having $k-1$ letters equal to $x$ and all of which begin with $x$.

Consider now the summand $f_{m-1}(i)\cdot c_m(n-i,k-1)$ in Equation (\ref{e2}).
 By the induction hypothesis, $c_m(n-i,k-1)$ equals the number of $\mathcal R$-restricted words of  length $n-i-1$ with $k-2$ letters equal to $x$. We first insert $x$ at the beginning of each such word.  In front of $x$, we insert an arbitrary word of length $i-1$ over $\alpha$, which are $f_{m-1}(i)$ in number. We thus obtain all words of length $n$ over $\alpha\cup\{x\}$, such that the first appearance of $x$ is in the $i$th  position.
Since the restriction $\mathcal R$ does not concern $x$, it follows that $x$ may be in an arbitrary place in a word. It implies that the right-hand side of Equation (\ref{e2}) counts all the desired words.
\end{proof}
\begin{remark}We stress the fact that the restriction $\mathcal R$ concerns only the letters from $\alpha$. Also, our result essentially depends on the fact that $f_0(1)=1$.
\end{remark}
Now, we illustrate our method by a simple example.
\begin{example}
Assume that $f_0(i)=1$ for $i=1,2,\ldots$. According to Equation (\ref{eq3}), we have $f_{m-1}(n)=m^{n-1}$. Hence $f_{m-1}(n)$ equals the number of words of length $n-1$ over the alphabet $\{0,1,\ldots,m-1\}$ with no restriction. Then, $c_1(n,k)={n-1\choose k-1}.$
This means that $C_1(n)=L_n$. It follows that  $C_m(n)=L_n^m.$
From the well known formula for the terms of $L_n^m$, we obtain
\begin{equation}\label{bu}c_m(n,k)=m^{n-k}{n-1\choose k-1}.\end{equation}
Equation (\ref{bu}) is in accordance with Proposition \ref{hr}. Namely, according to Proposition \ref{hr},  $c_m(n,k)$ equals the number of words of length $n-1$ over $\{0,1,\ldots,m\}$ with $k-1$ letters equal to $m$, and with no restrictions. These $k-1$ letters may be chosen in ${n-1\choose k-1}$ ways. The remaining letters may be arbitrary letters from $\{0,1,\ldots,m-1\}$, which are   $m^{n-k}$ in number.
As a byproduct, using Equation (\ref{e3}), we obtain the following binomial identity:
\begin{identity}
For $m>1$, we have
\begin{equation}\label{cp}m^{n-k}{n-1\choose k-1}=\sum_{j=0}^{n-k}(m-1)^{j}{n-1\choose k+j-1}{k+j-1\choose j}.\end{equation}
\end{identity}
This simple case is related with to Tchebychev polynomials $U_n(x)$ of the second kind.
\begin{corollary}
  The  expression $\vert[x^{n-k}](U_{n+k-2}(x))\vert$ equals the number of words of length $n-1$ over the alphabet $\{0,1,2\}$ having $k-1$ twos.
\end{corollary}
\begin{proof}
 It is a well-known fact that  $(-1)^k2^{n-k}{n-1\choose k-1}$ is the coefficient of $U_{n+k-2}(x)$ by $x^{n-k}$. We thus obtain
 \[[x^{n-k}](U_{n+k-2}(x))=(-1)^{n-k}\sum_{j=0}^{n-k}{n-1\choose k+j-1}{k+j-1\choose j}.\]
This is the case, when in Equation (\ref{cp}), we take $m=2$.
\end{proof}
\end{example}

\section{More examples}

Firstly,  we revise the result from the author~\cite[Corollary 9]{ja2}.
\begin{example}\label{ex1}
We define $f_0(1)=f_0(1)=1$, and $f_0(n)=0$ otherwise. According to the author~\cite[Corollary 33]{ja1}, we know that $f_{m-1}(n)$ equals the number of words of length $n-1$ over the alphabet $\{0,1,\ldots,m-1\}$ having all zeros isolated, which is the restriction $\mathcal R$.

\begin{corollary} The number $c_m(n,k)$ equals the number of words of length $n-1$ over the alphabet $\{0,1,\ldots,m\}$ which have $k-1$ letters equal to $m$ and all zeros are isolated.
Also,
\[c_1(n,k)={k\choose n-k},\] and
\[c_m(n,k)=\sum_{j=\lceil\frac{n}{2}\rceil-k}^{n-k}(m-1)^{j}{j+k-1\choose
k-1}{j+k\choose n-j-k},\left(m>1,\left\lceil\frac{n}{2}\right\rceil\geq k\right).\]
\end{corollary}
\begin{proof}
The first statement follows from Proposition \ref{hr}.
From the author~\cite[Corollary 10]{ja2}, it follows that $c_1(n,k+j)={k+j\choose n-k-j},(j=0,\ldots,n-k).$ This implies that $k+j\geq n-k-j$, which yields $2j\geq n-2k$ and $n\geq 2k$.
The formula is true according  to Equation (\ref{e3}).
\end{proof}
\end{example}
Next, we reexamine the result in the author~\cite[Corollary 28]{ja1}.
\begin{example}\label{ex2}
We define $f_0(n)=1$ when $n$ is odd, and $f_0(n)=0$ otherwise.
In this case, $f_m(n)$ equals the number of words of length $n-1$ over the alphabet $\{0,1,\ldots,m\}$, avoiding runs of zeros of odd lengths.
From the author~\cite[Proposition 24]{ja2}, it follows that
\[c_1(n,k)=\begin{cases}{\frac{n-k}{2}+k-1\choose k-1},&\text{ if $n-k$ is even};\\
0,&\text{ if $n$ is odd}.\end{cases}\]

The number $c_1(n,k)$ equals the number of  binary words of length $n-1$ with $k-1$ ones, avoiding runs of zeros of odd lengths. This follows from bijection (\ref{kor}).
 We add a short direct proof.
\begin{corollary} The number  $c_1(n,k)$ equals the number of
binary words of length $n-1$ with $k-1$ ones, avoiding runs of zeros of odd lengths.
\end{corollary}
 \begin{proof} Assume that  $n$ and $k$ are of different parities.
  Since a word of length $n-1$ with $k-1$ ones must have $n-k$ zeros, and since  $n-k$ is odd,  we conclude that such a word must have an odd run of zeros. It follows that $c_1(n,k)=0$.
 If $n$ and $k$ are of the same parity, then $n-k$ is even. This means that there are $\frac{n-k}{2}$ pairs of zeros. Of these $\frac{n-k}{2}$ pairs and $k-1$ ones, we may form ${\frac{n-k}{2}+k-1\choose k-1}$ words of length $n-1$ having $k-1$ ones and avoiding runs of zeros of odd lengths.
 \end{proof}
 Using induction and Proposition \ref{hr}, we obtain
 \begin{corollary} The number  $c_m(n,k)$ equals the number of words of length $n-1$ over $\{0,1,\ldots,m\}$ with $k-1$ letters equal to $m$, avoiding runs of zeros of odd lengths.
 \end{corollary}
From  Equation (\ref{e3}), we  obtain an explicit formula for $c_m(n,k)$.

 \end{example}
 \begin{example}\label{ex3}
We define $f_0(i)=i,(i=1,2,\ldots)$. The words of length $n-1$, avoiding $01$, are $\{11\ldots 1,11\cdots10,\ldots,00\ldots 0\}$. Hence, there are $n$ such words.
Applying Proposition \ref{hr} several times, we obtain
\begin{corollary}
 The number $c_m(n,k)$ equals the number of words of length $n-1$ over $\{0,1,\ldots,m+1\}$ having $k-1$ letters equal to $m+1$ and avoiding $01$.
\end{corollary}

\begin{corollary} The following formula holds
\[c_1(n,k)={n+k-1\choose 2k-1}.\]

Also,
\[c_m(n,k)=\sum_{i=k}^n(m-1)^{i-k}{i-1\choose k-1}{n+i-1\choose 2i-1}.\]
\end{corollary}
\begin{proof} For the first part, we use induction on $k$. For $k=1$, we have $c_1(n,1)=\sum_{i=1}^ni c_1(n-i,0)=n$, since
$c_1(n-i,0)=1$ when $n=i$, and $c_1(n-i,0)=0$ otherwise. This means that the statement is true for $k=1$. Using induction reduces the problem to the following identity:
\[{n+k-1\choose 2k-1}=\sum_{i=1}^{n-k+1}i\cdot{n+k-i-2\choose 2k-3},(k>1).\]
We let $X$ denote the right-hand side of this equation. It follows that
\[X=\sum_{i=1}^{n-k+1}{n+k-i-2\choose 2k-3}+\sum_{i=2}^{n-k+1}{n+k-i-2\choose 2k-3}
+\cdots+\sum_{i=n-k+1}^{n-k+1}{n+k-i-2\choose 2k-3}.\]
Applying the horizontal recurrence for the binomial coefficients on each summand yields
\[X={n+k-2\choose 2k-2}+{n+k-3\choose 2k-2}+{2k-2\choose 2k-2}.\]
Using the horizontal recurrence once more proves the statement.

The formula is true according to Equation (\ref{e3}).
\end{proof}
Since ${n+k-1\choose 2k-1}$ obviously equals the number of binary words of length $n+k-1$ with $2k-1$ zeros, we obtain the following Euler type identity:
\begin{identity}
The number of binary words of length $n+k-1$ with $2k-1$ zeros equals
the number of ternary words of length $n-1$, having $k-1$ letters equal to $2$ and avoiding $01$.
\end{identity}
\end{example}
The final two examples concern the case $f_0(1)=0$. Note that in these cases, Proposition \ref{hr} can not be used. The first example is an extension of the result in the author~\cite[Proposition 13]{ja2}.
\begin{example}\label{ex4}
  We define $f_0(1)=0$, and $f_0(n)=1$ otherwise.
  It follows from the author~\cite[Corollary 24]{ja1} that, for $n>3$,  $f_m(n)$ equals the number of words of length $n-3$ over $\{0,1,\ldots,m\}$, where no two consecutive letters are nonzero. From the author~\cite[Proposition 13]{ja2}, we obtain $c_1(n,k)={n-k-1\choose k-1}$ for $\left(1\leq k\leq \left\lfloor\frac{n}{2}\right\rfloor\right)$, and $c_1(n,k)=0$ otherwise. Equation (\ref{cmk}) implies that $c_m(n,k)=0$ when $k> \left\lfloor\frac{n}{2}\right\rfloor$.
  \begin{corollary}
For $n>3$ and $1\leq k\leq \left\lfloor\frac{n}{2}\right\rfloor$, the number  $c_m(n,k)$ equals the number of words of length $n-3$ over $\{0,1,\ldots,m\}$ with $k-1$ ones, and where all nonzero letters are isolated.
An explicit formula for $c_m(n,k)$ is
\[c_m(n,k)=\sum_{j=0}^{\left\lfloor\frac{n}{2}\right\rfloor-k}(m-1)^{j}{j+k-1\choose
k-1}{n-k-j-1\choose k+j-1},\left(1\leq k\leq \left\lfloor\frac{n}{2}\right\rfloor\right),\]
and $c_m(n,k)=0$ when $k>\lfloor\frac{n}{2}\rfloor$.
\end{corollary}
\begin{proof}
We know that $c_1(n,k)$ equals the number of compositions of $n$ into $k$ parts, each of which is greater than $1$. Using the bijection (\ref{kor}), we conclude that, for $n>3$, $c_1(n,k)$ equals the number of binary words of length $n$
 beginning with $10$, ending with $0$ and where all ones are isolated. Omitting $10$ at the beginning, and $0$ at the end of each word, we conclude that $c_1(n,k),(n>3)$ equals the number of binary words of length $n-3$ with $k-1$ ones all of which  are isolated.
This means that the statement is true for $m=1$. Assume that it is true for $m-1$.
 In Equation (\ref{ekf}),
by the induction hypothesis, $c_{m-1}(n,k+t)$ equals the number of words of length $n-3$ with $k+t-1$ ones, where all nonzero letters are isolated. Replacing $t$ ones with $m$'s, we obtain the desired words. The number $t$ may be chosen in ${k+t-1\choose t}$ ways. Hence,  the right-hand side of Equation (\ref{ekf}) counts all the desired words.

 The formula follows from Equation (\ref{e3}) and the fact that $n-k-j-1\geq k+j-1$.
\end{proof}
\end{example}
\begin{example} Define $f_0$ in the following way: $f_0(2)=f_0(3)=1$, and $f_0(n)=0$
otherwise. The author in the author~\cite[Proposition 5]{ja2} proved  that
$c_1(n,k)={k\choose n-2k},\left( \left\lceil\frac{n}{3}\right\rceil\leq k\leq \left\lfloor\frac{n}{2}\right\rfloor\right)$
and $c_1(n,k)=0$ otherwise.

We know that $c_1(n,k)$ equals the number of compositions of $n$ into $k$ parts equal to
either $2$ or $3$.
In other words, for $n\geq 3$ and  $1\leq k\leq \left\lfloor\frac{n}{2}\right\rfloor$,  $c_1(n,k)$ equals the number of binary words of length $n-1$ with $k-1$  ones, which have the following properties:
        a word begins with $0$ and ends with $0$. Also, zero avoids a run of length greater than $2$, and all ones are isolated.
\begin{corollary} For $n>3$, the number $c_m(n,k)$ equals the number of words of length $n-1$ over the alphabet $\{0,1,\ldots,m\}$ with $k-1$ ones, which begin and end with $0$. Also, $0$ avoids  a run of length greater than $2$ and all nonzero letters are isolated.
\end{corollary}
\begin{proof}
We showed that the statement is true for $m=1$.
Assume that the statement holds for $m-1$. Consider the term
${k+t-1\choose t}c_{m-1}(n,k+t)$ in Equation (\ref{ekf}). The number  $c_{m-1}(n,k+t)$ equals the number of the desired words of length $n-1$ over $\{0,1,\ldots,m-1\}$ with $k+t-1$ ones. We replace $t$ of  $k+t-1$ ones with $m$ and then sum over $t$ to obtain $c_m(n,k)$.
\end{proof}
From Equation (\ref{cmk}), it follows that $c_m(n,k)=0$ if $k>\lfloor\frac{n}{2}\rfloor$.
 Otherwise, from Equation (\ref{e3}), we obtain
\[c_m(n,k)=\sum_{j=0}^{\lfloor\frac{n}{2}\rfloor}(m-1)^j{j+k-1\choose k-1}{k+j\choose n-2k-2j},\left(1\leq k\leq\left\lfloor\frac{n}{2}\right\rfloor\right).\]
\end{example}

\end{document}